\documentclass[a4paper]{article}

\usepackage{amsmath}
\usepackage{amsfonts}
\usepackage{amssymb}
\usepackage{amsthm}

\numberwithin{equation}{section}

\theoremstyle{plain}
\newtheorem{theorem}{Theorem}[section]

\newtheorem{lemma}[theorem]{Lemma}
\newtheorem{proposition}[theorem]{Proposition}
\newtheorem{corollary}[theorem]{Corollary}
\theoremstyle{definition}
\newtheorem{definition}[theorem]{Definition}
\theoremstyle{remark}
\newtheorem{remark}[theorem]{Remark}

\newcounter{counter_a}
\newcounter{counter_b}
\newenvironment{myenum}{\begin{list}{{\rm(\roman{counter_a})}}%
{\usecounter{counter_a}
\setlength{\itemsep}{0.5ex}\setlength{\topsep}{0.7ex}
\setlength{\leftmargin}{5ex}\setlength{\labelwidth}{5ex}}}{\end{list}}

\DeclareMathOperator\Real{Re}
\DeclareMathOperator\Imag{Im}
\renewcommand\Re{\Real}
\renewcommand\Im{\Imag}

\newcommand\cD{\mathcal D}

\newcommand\cG{\mathcal G}
\newcommand\cH{\mathcal H}
\newcommand\cK{\mathcal K}

\newcommand\cS{\mathcal S}
\newcommand\dC{\mathbb C}

\newcommand\CC{\mathbb C}
\newcommand\NN{\mathbb N}
\newcommand\RR{\mathbb R}

\newcommand\eps{\varepsilon}

\newcommand\ov{\overline}

\newcommand{\defeq}{\mathrel{\mathop:}=}

\newcommand\sess{\sigma_{\rm ess}}

\newcommand\LaD{\Delta_{D}^\Omega}
\newcommand\LaN{\Delta_{N}^\Omega}

\DeclareMathOperator\dom{dom}
\DeclareMathOperator\ran{ran}

\newcommand\void[1]{}

\title{A remark on Schatten--von Neumann properties of resolvent differences of generalized 
Robin Laplacians on bounded domains}

\author{
Jussi Behrndt \\
{\small \textit{Institut f\"ur Mathematik, Technische Universit\"at Berlin}} \\
{\small \textit{Stra\ss{}e des 17.\ Juni 136, D-10623 Berlin, Germany}} \\
{\small \textit{E-mail:} \texttt{behrndt@math.tu-berlin.de}}
\and
Matthias Langer \\
{\small\textit{Department of Mathematics and Statistics,
University of Strathclyde}} \\
{\small \emph{26 Richmond Street, Glasgow G1 1XH, United Kingdom}} \\
{\small \textit{E-mail:} \texttt{m.langer@strath.ac.uk}}
\and
Igor Lobanov \\[0.5ex]
Vladimir Lotoreichik \\[0.5ex]
Igor Yu.\ Popov \\
{\small \textit{ Department of Mathematics}} \\
{\small \textit{St.~Petersburg State University of
		Information Technologies, Mechanics and Optics\hspace*{-3ex}}} \\
{\small \textit{Kronverkskiy, 49,  St.\ Petersburg, Russia}} \\
{\small \textit{E-mails:} \texttt{lobanov.igor@gmail.com, vladimir.lotoreichik@gmail.com,}} \\
{\small \texttt{popov@mail.ifmo.ru}}
}

\begin{document}

\maketitle

\begin{center}
\textit{In memory of M.\,Sh.~Birman
(1928--2009)}
\end{center}

\begin{abstract}
In this note we investigate the asymptotic behaviour of the $s$-numbers of the 
resolvent difference of two generalized self-adjoint, maximal dissipative or 
maximal accumulative Robin Laplacians on a bounded domain $\Omega$ with smooth boundary $\partial\Omega$. 
For this we apply the recently introduced abstract notion of quasi boundary 
triples and Weyl functions from extension theory of symmetric operators together 
with Krein type resolvent formulae and well-known eigenvalue asymptotics of the 
Laplace--Beltrami operator on $\partial\Omega$. It will be shown that the 
resolvent difference of two generalized Robin Laplacians belongs to the 
Schatten--von Neumann class of any order $p$ for which
\begin{equation*}
p > \frac{\dim \Omega - 1}{3}\,.
\end{equation*}
Moreover, we also give a simple sufficient condition for the resolvent difference of two generalized
Robin Laplacians to belong to a Schatten--von Neumann class of arbitrary small order. 
Our results extend and complement classical theorems due to M.\,Sh.~Birman on Schatten--von Neumann properties
of the resolvent differences of Dirichlet, Neumann and self-adjoint Robin Laplacians.
\smallskip

%\noindent
%\textbf{Subject classification:} Primary 47A10 ; Secondary  81Q10, 35J10. \\

\end{abstract}

%----------------------------------------------------------------------
%
\section{Introduction}
\label{sec:intro}
It is well known that the difference of the resolvents of two self-adjoint
extensions of a symmetric operator (with equal infinite deficiency numbers)
usually behaves `better' than the resolvents themselves, e.g.\
even if the resolvents are non-compact operators, the difference may belong
to a Schatten--von Neumann class, or if the resolvents are from
a Schatten--von Neumann class, the difference may lie in one of smaller order.
In particular, according to classical results due to M.\,Sh.~Birman~\cite{Bi}  
the resolvent difference
of the Dirichlet and Neumann Laplacian in a bounded or unbounded domain $\Omega$
with compact $C^{\infty}$ boundary $\partial\Omega$ satisfies
\begin{equation*}
  (-\LaD-\lambda)^{-1} - (-\LaN-\lambda)^{-1}
  \in \cS_p(L^2(\Omega)),
  \qquad \forall\, p > \frac{\dim \Omega - 1}{2},
\end{equation*}
where $\cS_p(L^2(\Omega))$ is the Schatten--von Neumann class of order $p$
and $\LaD$, $\LaN$ are the Dirichlet and Neumann Laplacians on $\Omega$, respectively.
Analogous estimates were also obtained for the difference of the resolvents of self-adjoint 
Laplacians with (ordinary) Robin boundary conditions $\beta f\vert_{\partial\Omega}=\frac{\partial f}{\partial\nu}$,
where $\beta$ is a real-valued function on $\partial\Omega$ and 
$\frac{\partial}{\partial \nu}$ denotes the outer normal
derivative. Later such results on spectral asymptotics were refined and generalized by, e.g.\
M.\,Sh.~Birman and M.\,Z.~Solomjak in~\cite{BiS} and G.~Grubb in \cite{Gr-1}.
Recently some new Schatten--von Neumann properties of resolvent differences of differential operators
were announced by F.~Gesztesy and M.\,M.~Malamud in~\cite{GMa}, and
in the paper by G.~Grubb~\cite{Gr-5} the influence of generalized Robin boundary conditions
on the essential spectrum
in exterior domains was studied.

The main objective of the present paper is to extend and complement some results on
Schatten--von Neumann properties for the resolvent difference of self-adjoint Laplacians
from \cite{Bi}. Instead of Dirichlet, Neumann and self-adjoint Robin Laplacians we study
so-called \emph{generalized} Robin Laplacians which are self-adjoint, maximal dissipative or maximal accumulative.
More precisely, we study self-adjoint, maximal dissipative and maximal accumulative realizations $-\Delta_{\Theta_1}^{\Omega}$ and $-\Delta_{\Theta_2}^{\Omega}$ of the
Laplacian corresponding to the generalized (or non-local) Robin boundary conditions 
\[
  \Theta_1 \frac{\partial f}{\partial \nu}\Big|_{\partial\Omega} = f\big|_{\partial\Omega} \qquad\text{and}\qquad
  \Theta_2 \frac{\partial f}{\partial \nu}\Big|_{\partial\Omega} = f\big|_{\partial\Omega},
\]
respectively, where $\Theta_1$ and $\Theta_2$ are self-adjoint, maximal dissipative or maximal accumulative operators in $L^2(\partial\Omega)$
such that $0 \notin \sess(\Theta_i)$, $i=1,2$. We note that generalized self-adjoint Robin Laplacians were recently also considered by F.~Gesztesy and M.~Mitrea
in \cite{GMi-1,GMi-2,GMi-3,GMi-4}. It is shown in Theorem~\ref{thm:1} and Corollary~\ref{cor:1} that 
\begin{equation}
\label{eq:robindiff1}
  (-\Delta_{\Theta_1}^{\Omega} -\lambda)^{-1} - (-\Delta_{\Theta_2}^{\Omega} - \lambda)^{-1} \in \cS_p(L^2(\Omega)),
  \qquad \forall\, p > \frac{\dim \Omega - 1}{3},
\end{equation}
holds for all $\lambda \in \rho(-\Delta_{\Theta_1}^{\Omega})\cap\rho(-\Delta_{\Theta_2}^{\Omega})$.
Moreover, if $\Theta_1-\Theta_2 \in \cS_{p_0}(L^2(\Omega))$
for some $p_0\in (0,\infty)$, then
\begin{equation}
\label{eq:robindiff2}
  (-\Delta_{\Theta_1}^{\Omega} -\lambda)^{-1} - (-\Delta_{\Theta_2}^{\Omega} - \lambda)^{-1} \in \cS_p(L^2(\Omega)),
  \quad \forall\, p > \frac{(\dim \Omega -1)p_0}{(\dim \Omega - 1)+3p_0 }\,;
\end{equation}
see Theorem~\ref{thm:3}. The proofs of these estimates are quite elementary and short when applying the abstract 
concept of quasi boundary triples and Weyl functions from extension theory of symmetric operators together
with Krein type resolvent formulae from \cite{BL} and well-known eigenvalue asymptotics of the Laplace--Beltrami
operator on $\partial\Omega$; see, e.g.\ \cite{Ag}. 
We note that our main results \eqref{eq:robindiff1} and
\eqref{eq:robindiff2} can be proved in the same way for generalized Robin Schr\"{o}dinger operators $-\Delta_{\Theta_i}^{\Omega}+V$
with a real valued $L^\infty$ potential $V$ or for more general
uniformly elliptic differential operators with coefficients satisfying appropriate conditions.

% ********************************************************************
\section{Quasi boundary triples}

In this section we briefly recall the abstract notion of quasi boundary triples and Weyl functions in extension theory of symmetric operators, 
some of their properties and how
they can be applied to the Laplacian on bounded domains.
This concept  was introduced in connection with elliptic boundary value problems 
by the first two authors in \cite{BL} as a generalization of the
notion of ordinary and generalized boundary triples from \cite{BrGeP,DHMaSn,DMa-1,DMa-2,Ma}.
The following definition is a variant of \cite[Definition 2.1]{BL} for densely defined, closed,
symmetric operators.

% -------------------------------------------------------------------
\begin{definition}
\label{def:qbt}
Let $A$ be a densely defined, closed, symmetric operator in a Hilbert space $\cH$.
We say that $(\cG, \Gamma_0, \Gamma_1)$ is a \emph{quasi boundary triple}
for $A^*$ if $\cG$ is a Hilbert space, $\Gamma_0$ and $\Gamma_1$ are linear
mappings defined on the same subset $\dom\Gamma_0=\dom\Gamma_1$ of $\dom A^*$ with values in $\cG$
such that $T \defeq A^*|_{\dom\Gamma_0}$ satisfies $\ov{T} = A^*$,
that $\binom{\Gamma_0}{\Gamma_1}\colon\dom T \to \cG\times\cG$ has dense range,
that $A_0 \defeq T|_{\ker\Gamma_0}$ is self-adjoint and that the identity
\begin{equation*}
\label{g.i.}
  (Tf, g)_{\cH} - (f, Tg)_{\cH} = (\Gamma_1 f, \Gamma_0 g)_{\cG} - (\Gamma_0 f, \Gamma_1 g)_{\cG}
\end{equation*}
holds for all $f, g \in \dom T$.
\end{definition}

From the definition it follows that both $\ran \Gamma_0$ and $\ran \Gamma_1$ are
dense in $\cG$.  Moreover, one can easily show that $\Gamma_0|_{\ker{(T-\lambda)}}$
is bijective from $\ker(T-\lambda)$ onto $\ran\Gamma_0$ for $\lambda\in\rho(A_0)$.
Next we recall the definition of the $\gamma$-field, the Weyl function and the 
parameterization of certain extensions of the symmetric operator $A$.

% -------------------------------------------------------------------
\begin{definition}
\label{def:gfield:weyl}
Let $A$ be a densely defined, closed, symmetric operator in a Hilbert space,
$(\cG,\Gamma_0,\Gamma_1)$ a quasi boundary triple for $A^*$ and $T$ as above.
\begin{myenum}
\item
The bijective mapping
\begin{equation*}
  \gamma(\lambda) \defeq (\Gamma_0|_{\ker{(T-\lambda)}})^{-1} \colon
  \ran \Gamma_0 \to \ker(T - \lambda), \qquad
  \lambda \in \rho(A_0),
\end{equation*}
is called \emph{$\gamma$-field}.
\item
The mapping
\begin{equation*}
M(\lambda) \defeq \Gamma_1 \gamma(\lambda), \qquad \lambda \in \rho(A_0),
\end{equation*}
is called \emph{Weyl function}.
\item
For a linear operator $\Theta$ in $\cG$, let $A_\Theta$ be the restriction of $T$ to
the set
\begin{equation*}
  \dom A_\Theta \defeq \big\{f\in \dom T \colon \Gamma_1 f = \Theta\Gamma_0 f\big\}.
\end{equation*}
\end{myenum}
\end{definition}

\noindent
We gather in one proposition some facts about the $\gamma$-field, the Weyl function
and $A_\Theta$ which were proved in \cite[Proposition~2.6 and Theorem~2.8]{BL}.

% -------------------------------------------------------------------
\begin{proposition}
\label{pro:gamma_Weyl}
Let $A$ be a densely defined, closed, symmetric operator in a Hilbert space and
let $(\cG, \Gamma_0, \Gamma_1)$ be a quasi boundary triple for $A^*$ with
$\gamma$-field $\gamma$ and Weyl function $M$.
For $\lambda \in \rho(A_0)$ the following assertions hold.
\begin{myenum}
\item % ----------
$\gamma(\lambda)$ is a densely defined bounded operator from $\cG$ to $\cH$
with $\dom \gamma(\lambda) = \ran \Gamma_0$.
\item % ----------
$\gamma(\overline{\lambda})^*$ is  a bounded mapping defined on $\cH$ with
values in $\ran \Gamma_1 \subset \cG$, and
\begin{equation}
\label{eq:gamma_bus}
  \gamma(\overline{\lambda})^* = \Gamma_1(A_0 - \lambda)^{-1}
\end{equation}
holds.
\item % ----------
$M(\lambda)$ maps $\ran \Gamma_0$ into $\ran \Gamma_1$. If, in addition, $T|_{\ker \Gamma_1}$ is
self-adjoint in $\cH$ and $\lambda \in \rho(T|_{\ker \Gamma_1})$, then $M(\lambda)$
maps $\ran \Gamma_0$ onto $\ran \Gamma_1$.
\item % ----------
For $\lambda \in \dC^+ \;(\text{or }\dC^-)$,
where $\dC^\pm \defeq \{z\in\dC\colon\pm \mathrm{Im}\,z>0\}$, the operator
\[
  \Im M(\lambda) \defeq \frac{1}{2i}\big(M(\lambda)-M(\lambda)^*\big)
\]
is bounded and positive (negative, respectively).
\item % ----------
Let $\Theta$ be a linear operator in $\cG$.  Then $\lambda$ is an eigenvalue of $A_\Theta$
if and only if $0$ is an eigenvalue of $\Theta-M(\lambda)$.  
If $\lambda$ is not an eigenvalue of $A_\Theta$, then Krein's formula
\begin{equation}
\label{krein_gen}
  (A_\Theta-\lambda)^{-1}f = (A_0-\lambda)^{-1}f
  + \gamma(\lambda)\big(\Theta-M(\lambda)\big)^{-1}\gamma(\ov\lambda)^*f
\end{equation}
holds for every $f\in\cH$ for which $\gamma(\ov\lambda)^*f \in \ran\big(\Theta-M(\lambda)\big)$.
\end{myenum}
\end{proposition}

In the following we recall how the concept of quasi boundary triples can be applied to
the Laplace operator on a bounded domain with $C^\infty$ boundary; cf. \cite[Section 4.2]{BL}.
We refer the reader to \cite{GMi-1,GMi-2,GMi-3,Gr-4} for recent work on the Laplacian and elliptic operators 
in non-smooth domains, and to \cite{BrGrW,Gr-1,GMa} for a different approach that leads to an ordinary boundary triple.
Let $\Omega\subseteq\RR^n$, $n>1$, be a bounded domain with $C^\infty$ boundary $\partial\Omega$, let $\nu(x)$ be 
the normal vector at the point $x\in\partial\Omega$ pointing outwards and
consider the differential expression $-\Delta$ on $\Omega$.
The operator $A$ defined by
\begin{equation*}
  Af = -\Delta f  , \qquad
  \dom A = H^2_0(\Omega) = \bigg\{ f \in H^2(\Omega) \colon
  f|_{\partial\Omega} = \frac{\partial f}{\partial \nu}\Big|_{\partial\Omega} = 0 \bigg\},
\end{equation*}
where $f|_{\partial\Omega}$ is the trace of $f$ and
\begin{equation*}
\frac{\partial f}{\partial \nu}\Big|_{\partial\Omega} = \sum_{i=1}^n \nu_i\frac{\partial f}{\partial x_i}\Big|_{\partial\Omega}
\end{equation*}
is the outer normal derivative, is a densely defined, closed, symmetric operator with equal infinite deficiency indices 
in $L^2(\Omega)$.  The adjoint of $A$ is
\begin{equation*}
  A^*f = -\Delta f, \qquad \dom A^* = \big\{f \in L^2(\Omega)\colon -\Delta f \in L^2(\Omega)\big\}.
\end{equation*}
We consider a restriction $T$ of $A^*$ so that we can define boundary mappings on $\dom T$.
As in \cite{BL} we use as domain of $T$ a Beals space, which turns out to be very convenient.
Let us recall its definition; for further details see, e.g.\ \cite{Be}.
Since $\partial\Omega$ is a $C^\infty$ boundary of $\Omega$, there exists $\eps_0 >0$ such that
for all $0 \le \eps < \eps_0$ the mapping $x \mapsto x -\eps \nu(x)$ is a homeomorphism
from $\partial\Omega$ onto $\{x - \eps \nu(x)\colon  x\in\partial\Omega\}$.
If $f \in L^2(\Omega)$ and $-\Delta f \in L^2(\Omega)$, then $f \in H^2_{\rm loc}(\Omega)$.
Hence $f_\eps$ defined by $f_\eps(x) \defeq f(x - \eps \nu(x))$ is in $L^2(\partial\Omega)$.
We say that $f$ has $L^2$ \emph{boundary value on} $\partial\Omega$
if $\lim_{\eps \to 0+} f_\eps$ exists as a limit in $L^2(\partial\Omega)$.
In this case we write $f|_{\partial\Omega} \defeq \lim_{\eps \to 0+} f_\eps$.

\begin{definition}
The \emph{Beals space of first order} is defined as
\begin{multline*}
  \cD_1(\Omega) \defeq \Big\{f \in L^2(\Omega) \colon -\Delta f \in L^2(\Omega),\text{and }\\
  f,\frac{\partial f}{\partial x_i}
  \text{ have $L^2$ boundary values on $\partial\Omega$ for all } i=1,\dots,n\Big\}.
\end{multline*}
\end{definition}

It is known (see \cite{Be}) that $H^2(\Omega) \subset \cD_1(\Omega) \subset H^\frac32(\Omega)$.
We define the operator $T$,
\[
  Tf = -\Delta f , \quad \dom T = \cD_1(\Omega),
\]
and the boundary mappings
\begin{equation*}
 \begin{split}
  \Gamma_0 \colon \dom T \mapsto L^2(\partial\Omega),& \qquad
    \Gamma_0 f = \frac{\partial f}{\partial \nu}\Big|_{\partial\Omega}, \\
  \Gamma_1 \colon \dom T \mapsto L^2(\partial\Omega),&\qquad \Gamma_1 f = f|_{\partial\Omega}.
 \end{split}
\end{equation*}
The restrictions
\[
  -\LaN \defeq T|_{\ker \Gamma_0}, \qquad
  -\LaD \defeq T|_{\ker \Gamma_1}
\]
are the usual Neumann and Dirichlet Laplacians whose domains are both contained
in $H^2(\Omega)$; moreover, $T|_{\ker \Gamma_0\cap\ker\Gamma_1} = A$.
Fundamental properties of Beals spaces imply that
\begin{equation*}
\ran \Gamma_0 = L^2(\partial\Omega), \quad \quad \ran \Gamma_1 = H^1(\partial\Omega).
\end{equation*}
In \cite{BL} it was shown that the triple $(L^2(\partial\Omega), \Gamma_0, \Gamma_1)$
is a quasi boundary triple for $A^*$.

In the next proposition Krein's formula is recalled, and a class of self-adjoint, maximal dissipative
and maximal accumulative generalized Robin Laplacians is parameterized with the help of the quasi 
boundary triple $(L^2(\partial\Omega), \Gamma_0, \Gamma_1)$. Recall that a linear operator $\Theta$ in
a Hilbert space is said to be \emph{dissipative} 
(\emph{accumulative}) if $\Im(\Theta f,f)\geq 0$
($\Im(\Theta f,f )\leq 0$, respectively) for all $f\in\dom\Theta$, and $\Theta$ is said to be \emph{maximal dissipative}
(\emph{maximal accumulative}) if $\Theta$ is dissipative (accumulative, respectively) and has no proper dissipative
(accumulative, respectively) extension. A dissipative (accumulative) operator $\Theta$ 
is maximal dissipative (maximal accumulative, respectively) if and only if $\Theta - \lambda_-$ ($\Theta - \lambda_+$, respectively) 
is surjective for some (and hence for all) $\lambda_-\in\CC^-$ ($\lambda_+\in\CC^+$, respectively).

% -------------------------------------------------------------------
\begin{proposition}
\label{pro:bl}
Let $T = -\Delta|_{\cD_1(\Omega)}$, $(L^2(\partial\Omega),\Gamma_0,\Gamma_1)$,
$\LaN$, $\LaD$ be as above and denote by $\gamma$ and $M$ the corresponding
$\gamma$-field and Weyl function.
Then the following assertions hold.
\begin{myenum}
\item
For $\lambda\in\CC\setminus\RR$, the operator $M(\lambda)$ is compact in $L^2(\partial\Omega)$
and $M(\lambda)^{-1}$ is a bounded operator from $H^1(\partial\Omega)$ onto $L^2(\partial\Omega)$.
\item
Krein's formula
\begin{equation}
\label{krein_DN}
  (-\LaD-\lambda)^{-1} - (-\LaN-\lambda)^{-1}
  = -\gamma(\lambda)M(\lambda)^{-1}\gamma(\ov{\lambda})^*
\end{equation}
holds for $\lambda\in\CC\setminus\RR$.
\setcounter{counter_b}{\value{counter_a}}
\end{myenum}
Further, let $\Theta$ be a self-adjoint (maximal dissipative, maximal accumulative) 
operator in $L^2(\partial\Omega)$ such that
$0\notin\sess(\Theta)$.  Then also the following statements are true.
\begin{myenum}
\setcounter{counter_a}{\value{counter_b}}
\item
For all $\lambda\in\CC\setminus\RR$ ($\lambda\in\CC^-$, $\lambda\in\CC^+$, respectively) the operator $\big(\Theta-M(\lambda)\big)^{-1}$ is bounded and everywhere defined
in $L^2(\partial\Omega)$.
\item
Denote by $-\Delta_\Theta^\Omega$ the restriction of $T$ to
\[
  \dom(-\Delta_\Theta^\Omega) = \big\{f\in\cD_1(\Omega) \colon \Gamma_1 f = \Theta \Gamma_0 f\big\}.
\]
Then $-\Delta_\Theta^\Omega$ is self-adjoint (maximal dissipative, maximal accumulative, respectively) 
in $L^2(\Omega)$, and Krein's formula
\[
  (-\Delta_\Theta^\Omega-\lambda)^{-1} - (-\LaN-\lambda)^{-1}
  = \gamma(\lambda)\big(\Theta-M(\lambda)\big)^{-1}\gamma(\ov\lambda)^*
\]
holds for $\lambda\in\CC\setminus\RR$ ($\lambda\in\CC^-$, $\lambda\in\CC^+$, respectively).
\end{myenum}
\end{proposition}

\begin{proof}
(i)
Without loss of generality let $\lambda\in\CC^+$.  That $M(\lambda)$ is compact
in $L^2(\partial\Omega)$ was proved in \cite[Proposition~4.6]{BL}.
Since
\[
  \Im\big(M(\lambda)x,x\big) = \big(\Im M(\lambda)x,x\big) >0
\]
for every $x\in L^2(\partial\Omega)$, $x\ne0$,
by Proposition~\ref{pro:gamma_Weyl}~(iv), we have $\ker M(\lambda)=\{0\}$.
It follows from the proof of \cite[Proposition~4.6]{BL} that $M(\lambda)$
is closed from $L^2(\partial\Omega)$ onto $H^1(\partial\Omega)$.
Hence its inverse $M(\lambda)^{-1}$ is also closed and by the
closed graph theorem bounded from $H^1(\partial\Omega)$ onto $L^2(\partial\Omega)$.

(ii)
In \eqref{krein_gen} we can choose $\Theta=0$, which yields
\eqref{krein_DN} applied to all $f$ for which $\gamma(\ov\lambda)^*f \in \ran M(\lambda)$.
It follows from \eqref{eq:gamma_bus} that
\[
  \ran\gamma(\ov\lambda)^* \subset \ran\Gamma_1 = H^1(\partial\Omega)
  = \ran M(\lambda),
\]
and hence Krein's formula \eqref{krein_DN} holds on the whole space $L^2(\Omega)$.

(iii) and (iv) were shown in \cite[Theorems~4.8 and 4.10]{BL}.
\end{proof}

% ********************************************************************
\section{Schatten--von Neumann classes and resolvent differences}
\label{sec:main}

Let $\cH$ and $\cK$ be separable Hilbert spaces. We denote by $\cS_{\infty}(\cH,\cK)$ the class of compact operators from
$\cH$ to $\cK$. For $T\in\cS_{\infty}(\cH,\cK)$ the eigenvalues $s_k(T)$ of the non-negative compact operator 
$(T^*T)^\frac12$,
ordered non-increasingly and counted with multiplicites, are called \emph{$s$-numbers} of $T$.

\begin{definition}
Let $\cH$ and $\cK$ be separable Hilbert spaces.
For $p>0$, the Schatten--von Neumann class is defined by
\begin{equation*}
  \cS_p(\cH,\cK) \defeq \bigg\{ T \in \cS_\infty(\cH,\cK) \colon
  \sum_{k=1}^\infty (s_k(T))^p < \infty\bigg\}.
\end{equation*}
If $\cK=\cH$, we write $\cS_p(\cH)$ for $\cS_p(\cH,\cK)$, $0<p\le\infty$.
\end{definition}

The set $\cS_p(\cH,\cK)$ is an ideal for every $p$ with $0<p\le\infty$ and a normed
ideal if $1\le p\le\infty$.
In the following two lemmas we recall some well-known facts about $s$-numbers and Schatten--von~Neumann classes.
For the proofs see, e.g.\ Sections II.\S2.1, II.\S2.2, III.\S7.2 in \cite{GoKr}.

% -------------------------------------------------------------------
\begin{lemma} \label{lem:snum}
\rule{0ex}{1ex}Let $\cH$ and $\cK$ be separable Hilbert spaces and let $T\in\cS_{\infty}(\cH,\cK)$. Then the following hold:
\begin{myenum}
\item
If $B$, $C$ are bounded operators, then
\[
  s_k(BTC) \le \|B\|\,\|C\|s_k(T) \qquad\text{for all }k\in\NN.
\]
\item
$s_k(T) = s_k(T^*)$ for all $k\in\NN$.
\item
If $s_k(T) = O(k^{-\alpha})$ as $k\to\infty$
for some $\alpha > 0$, then
\[
  T \in \cS_p(\cH,\cK) \quad \text{for all }\, p > \frac1\alpha\,.
\]
\end{myenum}
\end{lemma}

\begin{lemma} \label{lem:sp}
\rule{0ex}{1ex}
Let $\cH_0,\cH_1,\dots,\cH_n$ be separable Hilbert spaces,
let $p,p_1,\dots,p_n > 0$ be such that
\[
  \frac{1}{p} = \frac{1}{p_1} + \dots + \frac{1}{p_n}\,,
\]
and assume that $T_i$ are compact operators in $\cS_{p_i}(\cH_{i-1},\cH_i)$, $i=1,\dots,n$. Then
\[
  T_n\cdots T_1 \in \cS_p(\cH_0,\cH_n).
\]
\end{lemma}

The next lemma will be used in the proofs of our main results. 

% -------------------------------------------------------------------
\begin{lemma} \label{lem:snumbers}
Let $\Omega\subseteq\RR^n$ be a compact domain with $C^\infty$ boundary $\partial\Omega$.
Further, let $B$ be an everywhere defined, bounded operator from $L^2(\Omega)$
to $H^{r_1}(\partial\Omega)$ with $\ran B \subseteq H^{r_2}(\partial\Omega)$
for $r_2 > r_1 \ge 0$.
Then
\[
  B \in \cS_p\big(L^2(\Omega),H^{r_1}(\partial\Omega)\big)
  \quad \text{for all } \,p > \frac{n-1}{r_2-r_1}\,.
\]
\end{lemma}

\begin{proof}
As in \cite[Proposition~5.4.1]{Ag} we can define
\[
  \Lambda_{r_1,r_2} \defeq (I - \Delta_{\rm LB}^{\partial\Omega})^\frac{r_2-r_1}2,
\]
where $\Delta_{\rm LB}^{\partial\Omega}$ is the Laplace--Beltrami operator on $\partial\Omega$.
The operator $\Lambda_{r_1,r_2}$ is an isometric isomorphism from $H^{r_2}(\partial\Omega)$
onto $H^{r_1}(\partial\Omega)$.  The asymptotics of the eigenvalues of the Laplace--Beltrami
operator, $\lambda_k(\Delta_{\rm LB}^{\partial\Omega}) \sim Ck^{\frac{2}{n-1}}$
with some constant $C$, imply that
\[
  s_k(\Lambda_{r_1,r_2}^{-1}) = O\bigl(k^{-\frac{r_2 -r_1}{n-1}}\bigr),
  \qquad k\to\infty,
\]
where $\Lambda_{r_1,r_2}^{-1}$ is considered as an operator in $H^{r_1}(\partial\Omega)$.
We can write $B$ in the form
\[
  B = \Lambda_{r_1,r_2}^{-1}(\Lambda_{r_1,r_2} B).
\]
The operator $B$ is closed as an operator from $L^2(\Omega)$ to $H^{r_1}(\partial\Omega)$,
hence also closed as an operator from $L^2(\Omega)$ to $H^{r_2}(\partial\Omega)$,
which implies that it is bounded from $L^2(\Omega)$ to $H^{r_2}(\partial\Omega)$.
Therefore the operator $\Lambda_{r_1,r_2} B$ is bounded from $L^2(\partial\Omega)$ to
$H^{r_1}(\partial\Omega)$, and hence Lemma~\ref{lem:snum}~(i) implies
\[
  s_k(B) \le \|\Lambda_{r_1,r_2}B\|s_k(\Lambda_{r_1,r_2}^{-1})
  = O(k^{-\frac{r_2-r_1}{n-1}}), \qquad k\to\infty,
\]
from which the assertion follows by Lemma~\ref{lem:snum}~(iii).
\end{proof}

The next theorem, our first main result, is about Schatten--von Neumann properties
of differences of resolvents of the Neumann Laplacian and a Laplacian determined
by some boundary operator $\Theta$. For similar results involving the Dirichlet, Neumann and Robin Laplacian 
we refer the reader to  \cite{AlB,Bi,BiS,GMa,Gr-2,Gr-3,Gr-4}
and references therein.

% -------------------------------------------------------------------
\begin{theorem}
\label{thm:1}
Let $\Omega\subseteq\RR^n$ be a bounded domain with $C^\infty$ boundary $\partial\Omega$
and let $\Theta$ be a self-adjoint (maximal dissipative, maximal accumulative) operator in $L^2(\partial\Omega)$ such that $0\notin\sess(\Theta)$.
Denote by $-\LaN$ the Neumann Laplacian on $\Omega$ and by $-\Delta_\Theta^\Omega$ the generalized Robin Laplacian
from Proposition~{\rm\ref{pro:bl}~(iv)}.  Then 
\begin{equation}
\label{res_diff1}
  (-\Delta_\Theta^\Omega-\lambda)^{-1} - (-\LaN-\lambda)^{-1} \in \cS_p(L^2(\Omega))
  \quad\text{for all } \, p > \frac{n-1}{3}\,
\end{equation}
and all $\lambda\in\rho(-\Delta_\Theta^\Omega)\cap\rho(-\LaN)$.
In particular, for $n=2$ and $n=3$ the resolvent difference is a trace class operator.
\end{theorem}

\begin{proof}
According to Proposition~\ref{pro:bl}~(iv) we have Krein's formula
\begin{equation}
\label{krein_pr}
  (-\Delta_\Theta^\Omega-\lambda)^{-1} - (-\LaN-\lambda)^{-1}
  = \gamma(\lambda)\big(\Theta-M(\lambda)\big)^{-1}\gamma(\ov\lambda)^*
\end{equation}
for $\lambda\in\CC\setminus\RR$ ($\lambda\in\CC^-$, $\lambda\in\CC^+$, respectively).  Equation \eqref{eq:gamma_bus},
the inclusion $\dom(\LaN) \subseteq H^2(\Omega)$ and
the trace theorem (see, e.g.\ \cite{AdF,Wl}) imply that
\[
  \ran\big(\gamma(\ov\lambda)^*\big) \subseteq H^{\frac{3}{2}}(\partial\Omega).
\]
Because the operator $\gamma(\ov\lambda)^*$ is bounded from $L^2(\Omega)$ to
$L^2(\partial\Omega)$ by Proposition~\ref{pro:gamma_Weyl}~(ii), it is closed from
$L^2(\Omega)$ to $H^{\frac{3}{2}}(\partial\Omega)$ and hence bounded by the
closed graph theorem.  Now Lemma~\ref{lem:snumbers} yields
$\gamma(\ov\lambda)^* \in \cS_p(L^2(\Omega),L^2(\partial\Omega))$ for all
$p > \frac{2(n-1)}{3}$\,.

The same is true for $\gamma(\lambda)^*$, and hence the adjoint
$\gamma(\lambda) = \gamma(\lambda)^{**}$ is in
$\cS_p(L^2(\partial\Omega),L^2(\Omega))$ for all $p > \frac{2(n-1)}{3}$\,.
The operator $(\Theta-M(\lambda))^{-1}$ is bounded by Proposition~\ref{pro:bl}~(iii).
Therefore Lemma~\ref{lem:sp} implies that the right-hand side of \eqref{krein_pr}
is in $\cS_p(L^2(\Omega))$ for all $p > \frac{n-1}{3}$ and all $\lambda\in\CC\setminus\RR$ 
($\lambda\in\CC^-$, $\lambda\in\CC^+$, respectively). The fact that \eqref{res_diff1} holds 
for all points in $\rho(-\Delta_\Theta^\Omega)\cap\rho(-\LaN)$ follows from the formula
\begin{equation*}
\begin{split}
&(-\Delta_\Theta^\Omega-\mu)^{-1}-(-\LaN-\mu)^{-1}
=\bigl(I+(\mu-\lambda)(-\LaN-\mu)^{-1}\bigr)\\
&\qquad\times \bigl((-\Delta_\Theta^\Omega-\lambda)^{-1}-(-\LaN-\lambda)^{-1}\bigr)\,
\bigl(I+(\mu-\lambda)(-\Delta_\Theta^\Omega-\mu)^{-1}\bigr)
\end{split}
\end{equation*}
which is true for all $\lambda,\mu\in\rho(-\Delta_\Theta^\Omega)\cap\rho(-\LaN)$.
\end{proof}

Note that the resolvent of the Neumann Laplacian on a bounded domain itself is a compact operator,
so that the same holds true for the resolvent of the generalized Robin Laplacian $-\Delta_\Theta^\Omega$.
In other words, the spectrum of any self-adjoint (maximal dissipative, maximal accumulative) Robin Laplacian
$-\Delta_\Theta^\Omega$ in Theorem~\ref{thm:1} consists only of normal eigenvalues. Therefore, 
the intersections of the resolvent sets $\rho(-\Delta_{\Theta_1}^\Omega)\cap\rho(-\Delta_{\Theta_2}^\Omega)$
of two such Laplacians is always non-empty and
by taking the difference of the expressions in \eqref{res_diff1} 
we obtain the following corollary.

% -------------------------------------------------------------------
\begin{corollary}
\label{cor:1}
Let $\Theta_1$ and $\Theta_2$ be self-adjoint, maximal dissipative or maximal accumulative operators in $L^2(\partial\Omega)$ such that
$0\notin\sess(\Theta_i)$, $i=1,2$.  Then 
\[
  (-\Delta_{\Theta_1}^\Omega-\lambda)^{-1} - (-\Delta_{\Theta_2}^\Omega-\lambda)^{-1} \in \cS_p(L^2(\Omega))
  \quad\text{for all } \, p > \frac{n-1}{3}\,
\]
and all $\lambda\in\rho(-\Delta_{\Theta_1}^\Omega)\cap\rho(-\Delta_{\Theta_2}^\Omega)$.
\end{corollary}

\medskip

% -------------------------------------------------------------------
\begin{remark}
Proposition~\ref{pro:bl}~(iii), (iv) and hence Theorem~\ref{thm:1} are still valid if
$\Theta$ is a self-adjoint (maximal dissipative, maximal accumulative) linear relation (i.e.\ a multi-valued operator)
in $L^2(\partial\Omega)$ such that $0\notin\sess(\Theta)$; see \cite[Section~4]{BL}.
In particular, if $0\in\rho(\Theta)$, then $\Theta^{-1}$ is a bounded, self-adjoint (maximal dissipative, maximal accumulative, respectively) operator.
Conversely, for every bounded, self-adjoint (maximal dissipative, maximal accumulative) operator $B$,
 the inverse $B^{-1}$ is a
self-adjoint (maximal dissipative, maximal accumulative, respectively) relation with $0\in\rho(B^{-1})$.  Hence the restriction $-\Delta_{B^{-1}}^\Omega$
of $T$ to the domain
\[
  \dom(-\Delta_{B^{-1}}^\Omega)=\bigg\{f\in\cD_1(\Omega)\colon \frac{\partial f}{\partial n}\Big|_{\partial\Omega}
  = Bf|_{\partial\Omega}\bigg\}
\]
is a self-adjoint (maximal dissipative, maximal accumulative, respectively) realization of the Laplacian and satisfies
\[
  (-\Delta_{B^{-1}}^\Omega-\lambda)^{-1} - (-\LaN-\lambda)^{-1} \in \cS_p(L^2(\Omega))
  \quad\text{for all } p > \frac{n-1}{3}\,.
\]
As a special case we can treat (ordinary) Robin boundary conditions
\[
  \frac{\partial f}{\partial n}\Big|_{\partial\Omega} = \beta f|_{\partial\Omega},
\]
where the values of $\beta\in L^\infty(\partial\Omega)$ are real (have positive/negative imaginary parts, respectively).
\hfill$\bullet$
\end{remark}

Theorem~\ref{thm:1} does not cover the case of the difference of Dirichlet and
Neumann Laplacians since for the Dirichlet Laplacian we have to choose $\Theta=0$,
which does not satisfy $0\notin\sess(\Theta)$.
However, we obtain the following result, which is due to Birman \cite{Bi}.

% -------------------------------------------------------------------
\begin{theorem} \label{thm:2}
Let $\Omega\subseteq\RR^n$ be a bounded domain with $C^\infty$ boundary $\partial\Omega$.
Then
\begin{equation}
\label{res_diff2}
  (-\LaD-\lambda)^{-1} - (-\LaN-\lambda)^{-1} \in \cS_p(L^2(\Omega))
  \quad\text{for all }\, p > \frac{n-1}{2}\,
\end{equation}
and all $\lambda\in\rho(-\LaD)\cap\rho(-\LaN)$.
In particular, for $n=2$ the resolvent difference is a trace class operator.
\end{theorem}

\begin{proof}
By Proposition~\ref{pro:bl}~(ii) we have
\[
  (-\LaD-\lambda)^{-1} - (-\LaN-\lambda)^{-1}
  = -\gamma(\lambda)M(\lambda)^{-1}\gamma(\ov{\lambda})^*
\]
for $\lambda\in\CC\setminus\RR$.  The operator $\gamma(\ov\lambda)^*$ is bounded as an
operator from $L^2(\Omega)$ to $H^{\frac{3}{2}}(\partial\Omega)$; see the
proof of Theorem~\ref{thm:1}.  As an operator from $L^2(\Omega)$ to $H^1(\partial\Omega)$
it is in $\cS_p(L^2(\Omega),H^1(\partial\Omega))$ for all $p > 2(n-1)$ according
to Lemma~\ref{lem:snumbers}.

By Proposition~\ref{pro:bl}~(i), the operator $M(\lambda)^{-1}$ is bounded from
$H^1(\partial\Omega)$ to $L^2(\partial\Omega)$ and therefore $M(\lambda)^{-1}\gamma(\ov{\lambda})^*\in\cS_p(L^2(\Omega),L^2(\partial\Omega))$ for all $p > 2(n-1)$.  As in the proof of Theorem~\ref{thm:1}
we have $\gamma(\lambda) \in \cS_p(L^2(\partial\Omega),L^2(\Omega))$ for
all $p > \frac{2(n-1)}{3}$\,.  Hence Lemma~\ref{lem:sp} implies that the
resolvent difference in \eqref{res_diff2} is in $\cS_p(L^2(\Omega))$ for all
\[
  p > \frac{1}{\frac{1}{2(n-1)} + \frac{3}{2(n-1)}} = \frac{n-1}{2}\,.
\]
The same argument as in the proof of Theorem~\ref{thm:1} shows that \eqref{res_diff2} holds also for all $\lambda\in\rho(-\LaD)\cap\rho(-\LaN)$.
\end{proof}

% -------------------------------------------------------------------
\begin{remark}
Comparing the result of Theorem \ref{thm:1} with the result of Theorem~\ref{thm:2}
we see that we have $\frac{n-1}{3}$ instead of $\frac{n-1}{2}$.
The explanation comes from the fact that $M(\lambda)$ is compact in $L^2(\partial\Omega)$,
and hence $M(\lambda)^{-1}$ is unbounded in $L^2(\partial\Omega)$
whereas for $\Theta$ as in Theorem \ref{thm:1} the operator
$(\Theta - M(\lambda))^{-1}$ is bounded in $L^2(\partial\Omega)$ .
\end{remark}

Combining Theorems~\ref{thm:1} and \ref{thm:2} we obtain the following corollary.

% -------------------------------------------------------------------
\begin{corollary}
Let $\Omega\subseteq\RR^n$ be a bounded domain with $C^\infty$ boundary $\partial\Omega$
and let $\Theta$ be a self-adjoint (maximal dissipative, maximal accumulative) operator in $L^2(\partial\Omega)$ such that $0\notin\sess(\Theta)$.
Denote by $-\Delta_\Theta^\Omega$ the generalized Robin Laplacian from Proposition~{\rm\ref{pro:bl}~(iv)}.
Then 
\begin{equation*}
  (-\Delta_\Theta^\Omega-\lambda)^{-1} - (-\LaD-\lambda)^{-1} \in \cS_p(L^2(\Omega))
  \quad\text{for all } \, p > \frac{n-1}{2}\,
\end{equation*}
and all $\lambda\in\rho(-\Delta_\Theta^\Omega)\cap\rho(-\LaD)$.
\end{corollary}

\medskip

For ordinary boundary triples the resolvent difference belongs to the
same Schatten--von Neumann class as the resolvent difference of the
operators which parameterize the extensions; see \cite[Theorem 2 and Corollary 4]{DMa-1}.
In the case of quasi boundary triples the situation is different.
In the next Theorem we assume that $\Theta_2 - \Theta_1 \in \cS_{p_0}(L^2(\partial\Omega))$
for some $p_0>0$ and investigate Schatten--von Neumann properties of the resolvent
difference of the generalized Robin Laplacians parameterized by $\Theta_1$ and $\Theta_2$.

% -------------------------------------------------------------------
\begin{theorem}
\label{thm:3}
Let $\Omega\subseteq\RR^n$ be a bounded domain with $C^\infty$ boundary $\partial\Omega$.
Further, let $\Theta_1$ and $\Theta_2$ be bounded self-adjoint, maximal dissipative or maximal accumulative 
operators in
$L^2(\partial\Omega)$ such that $0\notin\sess(\Theta_i)$, $i=1,2$, and
\[
  \Theta_1-\Theta_2 \in \cS_{p_0}(L^2(\partial\Omega))
\]
for some $p_0\in(0,\infty)$. Denote by $-\Delta_{\Theta_i}^\Omega$ the restriction of $T$ as
in Proposition~{\rm\ref{pro:bl}~(iv)}.  Then
\begin{equation}
\label{res_diff3}
\begin{aligned}
  (-\Delta_{\Theta_1}^\Omega-\lambda)^{-1} - (-\Delta_{\Theta_2}^\Omega-\lambda)^{-1}
  \in \cS_p(L^2(\Omega))& \\
  \quad\text{for all } &\,p > \frac{(n-1)p_0}{n-1+3p_0}\,
\end{aligned}
\end{equation}
and all $\lambda\in\rho(-\Delta_{\Theta_1}^\Omega)\cap\rho(-\Delta_{\Theta_2}^\Omega)$.
\end{theorem}

By Theorem~\ref{thm:1} and Corollary~\ref{cor:1} the difference of the resolvents of $-\Delta_{\Theta_1}^\Omega$ and $-\Delta_{\Theta_2}^\Omega$
is a trace class operator for $n=2$ and $n=3$ without any further assumptions on $\Theta_1-\Theta_2$.
If, in addition, $\Theta_1-\Theta_2 \in \cS_{p_0}(L^2(\partial\Omega))$ for some $p_0\in(0,\infty)$, then 
this also holds for $n=4$.

\begin{corollary}
Let the assumptions be as in Theorem~\ref{thm:3}. For $n\in\{2,3,4\}$ and all $p_0\in(0,\infty)$ the resolvent difference 
in \eqref{res_diff3} is a trace class operator. The same holds for $n>4$ and 
$p_0<\frac{n-1}{n-4}$.
\end{corollary}

\begin{proof}[Proof of Theorem~\ref{thm:3}]
Assume first that $\Theta_2$ is self-adjoint and that $\Theta_1$ is self-adjoint (maximal dissipative or maximal accumulative,
respectively).
According to Proposition~\ref{pro:bl}~(iv) we can write
\begin{align*}
  &(-\Delta_{\Theta_1}^\Omega-\lambda)^{-1} - (-\Delta_{\Theta_2}^\Omega-\lambda)^{-1} \\[1ex]
  &\quad= \gamma(\lambda)\Big[\big(\Theta_1-M(\lambda)\big)^{-1} - \big(\Theta_2-M(\lambda)\big)^{-1}\Big]
  \gamma(\ov\lambda)^* \\[1ex]
  &\quad= \gamma(\lambda)\big(\Theta_1-M(\lambda)\big)^{-1}(\Theta_2-\Theta_1)
  \big(\Theta_2-M(\lambda)\big)^{-1}\gamma(\ov\lambda)^*
\end{align*}
for $\lambda\in\CC\setminus\RR$ ($\lambda\in\CC^-$, $\lambda\in\CC^+$, respectively).
As in the proof of Theorem~\ref{thm:1} we have
\begin{multline*}
  \gamma(\lambda) \in \cS_p\big(L^2(\partial\Omega),L^2(\Omega)\big),\quad
  \gamma(\ov\lambda)^* \in \cS_p\big(L^2(\Omega),L^2(\partial\Omega)\big) \\[1ex]
  \text{for all } p > \frac{2(n-1)}{3}\,.
\end{multline*}
The operators $(\Theta_i-M(\lambda))^{-1}$ are bounded by Proposition~\ref{pro:bl}~(iii).
Hence, using Lemma~\ref{lem:sp} we obtain that the resolvent difference
in \eqref{res_diff3} is in $\cS_p(L^2(\Omega))$ for all
\[
  p > \frac{1}{\frac{3}{2(n-1)} + \frac{1}{p_0} + \frac{3}{2(n-1)}}
  = \frac{(n-1)p_0}{n-1+3p_0}\,.
\]
The same argument as in the proof of Theorem~\ref{thm:1} shows that \eqref{res_diff2} holds also for all
$\lambda\in\rho(-\Delta_{\Theta_1}^\Omega)\cap\rho(-\Delta_{\Theta_2}^\Omega)$.

In the case that $\Theta_1$ and $\Theta_2$ are both either maximal dissipative or
maximal accumulative the above arguments remain valid for $\lambda\in\CC^-$ or $\lambda\in\CC^+$, respectively,
and hence \eqref{res_diff3} holds also in this case.

Let us now consider the case that $\Theta_1$ is maximal dissipative and $\Theta_2$ is maximal accumulative.
If $\Theta_1$ is maximal accumulative and $\Theta_2$ is maximal dissipative a similar reasoning applies.
As $\Theta_1-\Theta_2\in\cS_{p_0}(L^2(\partial\Omega))$ we also have
\begin{equation*}
\Re(\Theta_1-\Theta_2)\in\cS_{p_0}(L^2(\partial\Omega))\quad\text{and}\quad
\Im(\Theta_1-\Theta_2)\in\cS_{p_0}(L^2(\partial\Omega)),
\end{equation*}
and since $\Im\Theta_1\geq 0$ and $\Im\Theta_2\leq 0$ we conclude from the inequalities
\begin{equation*}
0\leq \Im\Theta_1\leq\Im(\Theta_1-\Theta_2)\quad\text{and}\quad 0\leq -\Im\Theta_2\leq\Im(\Theta_1-\Theta_2)
\end{equation*}
that also $\Im\Theta_i$, $i=1,2$, belong to $\cS_{p_0}(L^2(\partial\Omega))$. Therefore 
\begin{equation*}\Theta_i-\Re\Theta_i\in\cS_{p_0}(L^2(\partial\Omega))\quad\text{and}\quad
\sess(\Theta_i)=\sess(\Re\Theta_i),\qquad i=1,2,
\end{equation*}
and by the first part of the proof each of the resolvent differences
\begin{equation}\label{sum}
\begin{split}
&(-\Delta_{\Theta_1}^\Omega-\lambda)^{-1} - (-\Delta_{\Re\Theta_1}^\Omega-\lambda)^{-1},\\
&(-\Delta_{\Re\Theta_1}^\Omega-\mu)^{-1} - (-\Delta_{\Re\Theta_2}^\Omega-\mu)^{-1},\\
&(-\Delta_{\Re\Theta_2}^\Omega-\vartheta)^{-1} - (-\Delta_{\Theta_2}^\Omega-\vartheta)^{-1}
\end{split}
\end{equation}
belongs to $\cS_{p}(L^2(\Omega))$, where $p>\frac{(n-1)p_0}{n-1+3p_0}$. Moreover, the resolvents of 
$-\Delta_{\Theta_i}^\Omega$ and $-\Delta_{\Re\Theta_i}^\Omega$, $i=1,2$,
are all
compact and hence almost all $\lambda\in\CC$ belong to the intersection of the resolvent sets of these
generalized Robin Laplacians. Then it follows from \eqref{sum} that the difference of the resolvents of $-\Delta_{\Theta_1}^\Omega$ and 
$-\Delta_{\Theta_2}^\Omega$ satisfies \eqref{res_diff3}.
\end{proof}

\subsection*{Acknowledgements}
M.~Langer was supported by the Engineering and Physical Sciences Research
Council (EPSRC) of the UK, grant EP/E037844/1.
I.~Lobanov, V.~Lotoreichik and I.~Popov were supported by grant 2.1.1/4215
of the programme ``Development of the potential of High School in Russia 2009--2010''.
V.~Lotoreichik was also supported by the personal grant 2.1/30-04/035 of
the government of St.~Petersburg and the DAAD Leonard Euler programme, grant 50077360.

%----------------------------------------------------------------------

\end{document}